\documentclass[a4paper,twoside,10pt]{article}

\newcommand{\mytitle}{Optimal Convergence Rates Results for Linear Inverse Problems in Hilbert Spaces}
\title{\mytitle}
\author%
{%
	V.~Albani$^1$\\{\footnotesize\href{mailto:vvla@impa.br}{vvla@impa.br}} \and 
	P.~Elbau$^1$\\{\footnotesize\href{mailto:peter.elbau@univie.ac.at}{peter.elbau@univie.ac.at}} \and 
	M.~V.~de~Hoop$^2$\\{\footnotesize\href{mailto:mdehoop@rice.edu}{mdehoop@rice.edu}} \and 
	O.~Scherzer$^{1,3}$\\{\footnotesize\href{mailto:otmar.scherzer@univie.ac.at}{otmar.scherzer@univie.ac.at}}
}
\date{}

\usepackage[backend=bibtex,maxnames=10,babel=hyphen,defernumbers=true,abbreviate=false,backref=true]{biblatex}
\DeclareFieldFormat[report]{title}{``#1''}
\DeclareFieldFormat[book]{title}{``#1''}

\usepackage{amsmath}

\usepackage{amssymb}

\usepackage[a4paper,centering,bindingoffset=0cm,hmargin=3.5cm,vmargin=2.5cm,includeheadfoot]{geometry}

\usepackage{titlesec}
\titleformat{\section}{\filcenter\sc\large}{\thesection.\;}{0em}{}
\titleformat{\subsection}[runin]{\bf}{\thesubsection.\;}{0em}{}[.]

\newpagestyle{headers}%
{%
	\headrule
	\sethead%
		[\thepage]%
		[\footnotesize\sc V.~Albani, P.~Elbau, M.~V.~de Hoop, O.~Scherzer]%
		[]%
		{}%
		{\footnotesize\sc{\mytitle}}%
		{\thepage}
}
\usepackage[pdftex,colorlinks=true,linkcolor=blue,unicode,pdfpagelabels]{hyperref}
\hypersetup
{
	pdftitle={\mytitle},
	pdfauthor={V. Albani, P. Elbau, M. de Hoop, O. Scherzer},
	pdfkeywords={Convergence Rates, Linear Inverse Problems, Variational Inequalities, Distance Function}
}

\usepackage[hyperref,amsmath,thmmarks]{ntheorem}
\usepackage{aliascnt}

\theoremlisttype{all}
\theoremstyle{break}
\theoremseparator{.}
\theoremheaderfont{\kern-1em\normalfont\bfseries} 

\newaliascnt{proposition}{theorem}
\newtheorem{proposition}[proposition]{Proposition}
\aliascntresetthe{proposition}

\newaliascnt{corollary}{theorem}
\newtheorem{corollary}[corollary]{Corollary}
\aliascntresetthe{corollary}

\newaliascnt{lemma}{theorem}
\newtheorem{lemma}[lemma]{Lemma}
\aliascntresetthe{lemma}

\theorembodyfont{\normalfont}
\newaliascnt{definition}{theorem}
\newtheorem{definition}[definition]{Definition}
\aliascntresetthe{definition}

\newaliascnt{notation}{theorem}
\newtheorem{notation}[notation]{Notation}
\aliascntresetthe{notation}

\newaliascnt{example}{theorem}
\newtheorem{example}[example]{Example}
\aliascntresetthe{example}

\theoremstyle{nonumberplain}
\theoremseparator{:}
\theoremheaderfont{\kern-1em\normalfont\itshape}
\newtheorem{remark}{Remark}

\theorembodyfont{\normalfont}
\theoremsymbol{\ensuremath{\square}}
\theoremstyle{nonumberbreak}
\newtheorem{proof}{Proof}

\usepackage{mathtools}

\usepackage{bbm}
\newcommand{\R}{\mathbbm R}

\newcommand{\N}{\mathbbm N}

\newcommand{\Ord}{\mathcal O}
\newcommand{\e}{\mathrm e}
\renewcommand{\d}{\,\mathrm d}

\DeclareMathOperator*{\argmin}{arg\,min}

\begin{document}

\maketitle
\begin{center}
\parbox[t]{0.49\textwidth}{\footnotesize
\hspace*{-1ex}$^1$Computational Science Center\\
University of Vienna\\
Oskar-Morgenstern-Platz 1\\
A-1090 Vienna, Austria\\[1em]
\hspace*{-1ex}$^3$Johann Radon Institute for Computational\\
\hspace*{1em}and Applied Mathematics (RICAM)\\
Altenbergerstra{\ss}e 69\\
A-4040 Linz, Austria
}
\hfill
\parbox[t]{0.49\textwidth}{\footnotesize
\hspace*{-1ex}$^2$Department of Computational and Applied\\
\hspace*{1em}Mathematics and\\
Department of Earth Science\\
Rice University\\
6100 Main Street\\
Houston, TX 77005, USA}
\end{center}

\begin{abstract}
In this paper, we prove optimal convergence rates results for regularisation methods for solving linear ill-posed 
operator equations in Hilbert spaces. The result generalises existing convergence rates results on optimality 
of~\cite{Neu97} to general source conditions, such as logarithmic source conditions. Moreover, we also 
provide optimality results under variational source conditions, extending the results of~\cite{AndElbHooQiuSch15},
and show the connection to approximative source conditions, introduced in~\cite{HofMat07}.
\end{abstract}

\section{Introduction}

Let $L:X\to Y$ be a bounded linear operator between two Hilbert spaces $X$ and $Y$. We are interested in finding the minimum-norm solution $x^\dag\in X$ of the equation
\[ Lx = y \]
for some $y\in\mathcal R(L)$, that is the element $x^\dag\in\{x\in X\mid Lx=y\}$ with the property $\|x^\dag\|=\inf\{\|x\|\mid Lx=y\}$. It is well-known that this minimal-norm solution exists and is unique, see for example \cite[Theorem 2.5]{EngHanNeu96}.

Since $y$ is typically not exactly know, but only an approximation $\tilde y\in Y$ with $\|y-\tilde y\|\le\delta$ is given, we are looking for a family $(x_\alpha(\tilde y))_{\alpha\ge0}$ of approximative solutions so that for every sequence $(\tilde y_k)_{k\in\N}$ converging to $y$, we find a sequence $(\alpha_k)_{k\in\N}$ of regularisation parameters such that $(x_{\alpha_k}(\tilde y_k))_{k\in\N}$ tends to the minimum-norm solution $x^\dag$.

A standard way to construct this family is by using Tikhonov regularisation:
\[ x_\alpha^{\textrm{Tik}}(\tilde y) = \argmin_{x\in X}(\|Lx-\tilde y\|^2+\alpha\|x\|^2), \]
where the minimiser can be explicitly calculated from the optimality condition and reads as follows:
\begin{equation}\label{eqTikReg}
x_\alpha^{\textrm{Tik}}(\tilde y) = (\alpha+L^*L)^{-1}L^*\tilde y.
\end{equation}

More generally, we want to analyse regularised solutions of the form
\begin{equation}
\label{eq:reg}
 x_\alpha(\tilde y) = r_\alpha(L^*L)L^* \tilde y
\end{equation}
with some appropriately chosen function $r_\alpha$, see for example~\cite{HofMat07}.

The aim of this paper is then to characterise for a given regularisation method, generated by a family $(r_\alpha)_{\alpha>0}$, the 
optimal convergence rate with which $x_\alpha(\tilde y)$ tends to the minimum-norm solution $x^\dag$. 
This convergence rate depends on the solution $x^\dag$, and we will give an explicit relation between the spectral projections 
of $x^\dag$ with respect to the operator $L^*L$ and the convergence rate; first in \autoref{seConvRateExact} for the convergence of $x_\alpha(y)$ with the exact data~$y$, and then in \autoref{seConvRateNoisy} for $x_\alpha(\tilde y)$ with noisy data $\tilde y$. This generalises existing convergence rates results of~\cite{Neu97} to general source conditions, such as logarithmic source conditions.

Afterwards, we show in \autoref{seVarIneq} that these convergence rates can also be obtained from variational inequalities and establish the optimality of these general variational source conditions, extending the results of \cite{AndElbHooQiuSch15}.
It is interesting to note that variational source conditions are equivalent to convergence rates 
of the regularised solutions, while the classical results in~\cite{Gro84} are not. 

Finally, we consider in \autoref{seApproxSourcCond} approximate source conditions, which relate the convergence rates of the regularised solutions to the decay rate of a distance function measuring how far away the minimum-norm solution is from the classical range condition, see \cite{HofMat07,FleHofMat11}. We can show that these approximate source conditions are indeed equivalent to the convergence rates.


\section{Convergence Rates for Exact Data}\label{seConvRateExact}
In the following, we analyse the convergence rate of the sequence $(x_\alpha(y))_{\alpha>0}$ with the exact data 
$y\in\mathcal R(Y)$ to the minimum-norm solution $x^\dag$ of $Lx=y$. 

We investigate regularisation methods of the form \eqref{eq:reg}, which are generated by functions satisfying the following properties.
\begin{definition}\label{deGenerator}
We call a family $(r_\alpha)_{\alpha>0}$ of continuous functions $r_\alpha:[0,\infty)\to[0,\infty)$ the 
generator of a regularisation method if
\begin{enumerate}
\item \label{enGeneratorBounded} 
      there exists a constant $\rho\in(0,1)$ such that
      \[ r_\alpha(\lambda)\le\min\left\{\frac1\lambda,\frac\rho{\sqrt{\alpha\lambda}}\right\}
      \quad\text{for every}\quad\lambda>0,\;\alpha>0, \]
\item \label{enGeneratorError}
      the error function $\tilde r_\alpha:[0,\infty)\to[0,\infty)$, defined by
      \begin{equation}\label{eqGeneratorError}
      \tilde r_\alpha(\lambda)=(1-\lambda r_\alpha(\lambda))^2,\quad\lambda\ge0,
      \end{equation}
      is decreasing.
\item \label{enGeneratorErrorReg}
      For fixed $\lambda\ge0$ the map $\alpha\mapsto\tilde r_\alpha(\lambda)$ is continuous and increasing, and 
\item \label{enGeneratorLimit}
      there exists a constant $\tilde\rho\in(0,1)$ such that
      \[ \tilde r_\alpha(\alpha)<\tilde\rho\quad\text{for all}\quad\alpha>0. \]
\end{enumerate}
\end{definition}

\begin{remark}
These conditions do not yet enforce that $x_\alpha(y)\to x^\dag$. To ensure this, we could additionally impose that $\tilde r_\alpha(\lambda)\to0$ for every $\lambda>0$ as $\alpha\to0$.
\end{remark}

Let us now fix the notation for the rest of the article.
\begin{notation}\label{noSetting}
 Let $L:X\to Y$ be a bounded linear operator between two real Hilbert spaces $X$ and $Y$, $y\in\mathcal R(Y)$, 
 and $x^\dag\in X$ be the minimum-norm solution of $Lx=y$. 

Moreover, we choose a generator $(r_\alpha)_{\alpha>0}$ of a regularisation method, introduce the family 
$(\tilde r_\alpha)_{\alpha>0}$ of its error functions, and the corresponding family of regularised solutions shall be given by 
\eqref{eq:reg}.

We denote by $A\mapsto E_A$ and $A\mapsto F_A$ the spectral measures of the operators $L^*L$ and $LL^*$, respectively, on all Borel sets $A \subseteq [0,\infty)$.

Next, we define the right-continuous and increasing function 
        \begin{equation}
         \label{eqSpectralFunction}
         \begin{aligned}
            e:[0,\infty) &\to \R\,,\\
              \lambda &\mapsto \|E_{[0,\lambda]}x^\dag\|^2.
         \end{aligned}
        \end{equation}

Moreover, if $f:(0,\infty)\to\R$ is a right-continuous, increasing, and bounded function, we write  
\[ \int_a^bg(\lambda)\d f(\lambda) = \int_{(a,b]}g(\lambda)\d\mu_f(\lambda) \]
for the Lebesgue--Stieltjes integral of $f$, where $\mu_f$ denotes the unique non-negative Borel measure defined by $\mu_f((\lambda_1,\lambda_2])=f(\lambda_2)-f(\lambda_1)$ and $g\in L^1(\mu)$.
\end{notation}

\begin{remark}
In this setting, we can write the error
\begin{equation}
 \label{eqErrorExact}
 x_\alpha(y)-x^\dag= r_\alpha(L^*L)L^*y - x^\dag = (r_\alpha(L^*L)L^*L-I)x^\dag
\end{equation}
according to spectral theory in the form
\begin{equation}\label{eqCrResidum} 
\|x_\alpha(y)-x^\dag\|^2 = \int_0^{\|L\|^2}\tilde r_\alpha(\lambda)\d e(\lambda).
\end{equation}
We want to point out here that it follows directly from the definition that the minimum-norm solution $x^\dag$ is in the orthogonal complement $\mathcal N(L)^\perp$ of the nullspace of $L$, and we therefore do not have to consider the point $\lambda=0$ in the integrals in equation \eqref{eqCrResidum}.
\end{remark}

We first want to establish a relation between the convergence rate of the regularised solution $x_\alpha(y)$ for exact data $y$ to the minimum-norm solution $x^\dag$ and the behaviour of the spectral function \eqref{eqSpectralFunction}.
\begin{proposition}\label{thCr}
We use \autoref{noSetting} and assume that there exist an increasing function $\varphi:(0,\infty)\to(0,\infty)$ and constants $\mu\in(0,1)$ and $A>0$ such that we have for every $\alpha > 0$ the inequality 
\begin{equation}
\label{eqQualification}
\varphi(\lambda)\tilde r_\alpha^\mu(\lambda) \le A\varphi(\alpha)\quad\text{for all}\quad\lambda>0.
\end{equation}

Then, the following two statements are equivalent:
\begin{enumerate}
 \item There exists a constant $C>0$ with
       \begin{equation}\label{eqCr}
         \|x_\alpha(y)-x^\dag\|^2 \le C\varphi(\alpha)\quad\text{for all}\quad\alpha>0.
       \end{equation}
 \item There exists a constant $\tilde C>0$ with
       \begin{equation}\label{eqCrSpectralDecay}
          e(\lambda) \le\tilde C\varphi(\lambda)\quad\text{for all}\quad\lambda>0.
       \end{equation}
\end{enumerate}
\end{proposition}
\begin{proof}
According to \autoref{deGenerator}~\ref{enGeneratorError} the error function $\tilde r_\alpha$ is decreasing and thus 
it follows together with \eqref{eqCrResidum} that for all $\alpha > 0$
\begin{equation}
 \label{eq:r_est}
 \tilde r_\alpha(\alpha) e(\alpha) = \tilde r_\alpha(\alpha) \int_0^\alpha \d e(\lambda) \le 
\int_0^\alpha\tilde r_\alpha(\lambda)\d e(\lambda) \le \|x_\alpha(y)-x^\dag\|^2\;.
\end{equation}
\begin{itemize}
 \item Let first \eqref{eqCr} hold. Then, it follows from \eqref{eq:r_est} that for all $\alpha > 0$ 
       \begin{equation}
        \label{eq:tilder}
        \tilde r_\alpha(\alpha) e(\alpha) \le C\varphi(\alpha).
       \end{equation}
       Now, we use \autoref{deGenerator}~\ref{enGeneratorBounded}, which gives that 
       \begin{equation*}
        \tilde r_\alpha(\alpha)=(1-\alpha r_\alpha(\alpha))^2\ge(1-\rho)^2>0.
       \end{equation*}
       Using this estimate in \eqref{eq:tilder} yields~\eqref{eqCrSpectralDecay} with $\tilde C = \frac{C}{(1-\rho)^2}>0$.
\item Conversely, let \eqref{eqCrSpectralDecay} hold. Since $\|x_\alpha(y)-x^\dag\|^2\le\|x^\dag\|^2$ (which follows from \eqref{eqCrResidum} with $\tilde r_\alpha\le1$), it is enough to check the condition~\eqref{eqCr} for all $\alpha\in(0,\|L\|^2]$.

      We use \eqref{eqCrResidum} and integrate the right hand side 
      by parts, see for example \cite[Theorem~6.2.2]{CarBru00} about the integration by parts for Lebesgue--Stieltjes integrals, and obtain that 
      \begin{equation}
      \label{eq:er}
              \|x_\alpha(y)-x^\dag\|^2 = \tilde r_\alpha(\|L\|^2)e(\|L\|^2) + \int_0^{\|L\|^2}e(\lambda)\d(-\tilde r_\alpha)(\lambda)
      \end{equation}
      We split up the integral on the right hand side into two terms:
      \begin{equation}
      \label{eq:split}
      \int_0^{\|L\|^2}e(\lambda)\d(-\tilde r_\alpha)(\lambda) = \int_0^\alpha e(\lambda)\d(-\tilde r_\alpha)(\lambda) + 
      \int_\alpha^{\|L\|^2}e(\lambda)\d(-\tilde r_\alpha)(\lambda).
      \end{equation}
      The first term is estimated by using that the function $e$ is increasing and by utilising the assumption \eqref{eqCrSpectralDecay}:
      \[ \int_0^\alpha e(\lambda)\d(-\tilde r_\alpha)(\lambda) \leq  e(\alpha) \int_0^\alpha \d(-\tilde r_\alpha)(\lambda) = e(\alpha) (1-\tilde r_\alpha(\alpha)) \le \tilde{C} \varphi(\alpha). \]
      The second integral term in \eqref{eq:split} is estimated by using the inequalities~\eqref{eqCrSpectralDecay} and~\eqref{eqQualification}:
      \begin{equation*}
      \begin{aligned}
      \int_\alpha^{\|L\|^2} e(\lambda)\d(-\tilde r_\alpha)(\lambda)
      &\leq \tilde{C} \int_\alpha^{\|L\|^2} \varphi(\lambda)\d(-\tilde r_\alpha)(\lambda)\\
      &= \tilde{C}
           \int_\alpha^{\|L\|^2}\varphi(\lambda)\tilde r_\alpha^\mu(\lambda)
           \frac1{\tilde r_\alpha^\mu(\lambda)}\d(- \tilde r_\alpha)(\lambda) \\
      &\le A\tilde{C} 
           \varphi(\alpha)\int_\alpha^{\|L\|^2}\frac1{\tilde r_\alpha^\mu(\lambda)}
           \d(-\tilde r_\alpha)(\lambda) \\
      &= \frac{A\tilde{C}}{1-\mu}\varphi(\alpha) 
          (\tilde r_\alpha^{1-\mu}(\alpha)-\tilde r_\alpha^{1-\mu}(\|L\|^2)) \\
      &\le \frac{A\tilde{C}\tilde\rho^{1-\mu}}{1-\mu}\varphi(\alpha),
      \end{aligned}
      \end{equation*}
      where we used \autoref{deGenerator}~\ref{enGeneratorLimit} in the last step.
      Inserting the two estimates in~\eqref{eq:split} and in~\eqref{eq:er}, we find with $e(\|L\|^2)=\|x^\dag\|^2$ that
      \begin{equation}
      \label{eq:af}
       \|x_\alpha(y)-x^\dag\|^2 \le \tilde r_\alpha(\|L\|^2)\|x^\dag\|^2 + \tilde C\varphi(\alpha) + \frac{A\tilde C\tilde\rho^{1-\mu}}{1-\mu}\varphi(\alpha).
      \end{equation}
	From \eqref{eqQualification}, we deduce further that
      \[ \tilde r_\alpha(\|L\|^2) \le 
         \frac{A^{\frac1\mu}}{\varphi^{\frac1\mu}(\|L\|^2)}\varphi^{\frac1\mu}(\alpha) 
         \le \frac{A^{\frac1\mu}\varphi^{\frac1\mu-1}(\alpha)}{\varphi^{\frac1\mu}(\|L\|^2)} \varphi(\alpha) \leq c\varphi(\alpha)\quad\text{with}\quad 
         c = \frac{A^{\frac1\mu}}{\varphi (\|L\|^2)}, \]
	since $\varphi$ is increasing and $\mu < 1$.

      Thus, we get from \eqref{eq:af} that 
      \[ \|x_\alpha(y)-x^\dag\|^2 \le C\varphi(\alpha) \]
      with $C = c \|x^\dag\|^2 + \tilde C + \frac{A\tilde C\tilde\rho^{1-\mu}}{1-\mu}$.
\end{itemize}
\end{proof}

\begin{remark}
The condition \eqref{eqQualification} with the choice $\mu=\frac12$ was already used in~\cite{FleHofMat11} and such a function~$\varphi$ was called a qualification of the regularisation method.
\end{remark}

\begin{example}\label{exTikhonov}
In the case of Tikhonov regularisation, given by~\eqref{eqTikReg}, we have $r_\alpha(\lambda)=\frac1{\alpha+\lambda}$ and therefore, we get for the error function $\tilde r_\alpha$, defined by \eqref{eqGeneratorError}, the expression $\tilde r_\alpha(\lambda)=\frac{\alpha^2}{(\alpha+\lambda)^2}$. So, clearly, $\tilde r_\alpha(\alpha)=\frac14$ and all the conditions of \autoref{deGenerator} are fulfilled.

\begin{enumerate}
\item\label{enTikhonovClassical}
To recover the classical equivalence results, see~\cite[Theorem 2.1]{Neu97}, we set $\varphi(\alpha)=\alpha^{2\nu}$ for some $\nu\in(0,1)$ and find that the condition \eqref{eqQualification} with $A=1$ is for every $\mu\ge\nu$ fulfilled, since we have
\[ \varphi(\lambda)\tilde r_\alpha^\mu(\lambda) = \frac{\alpha^{2\mu}\lambda^{2\nu}}{(\alpha+\lambda)^{2\mu}} \le \frac{\alpha^{2\mu-2\nu}}{(\alpha+\lambda)^{2\mu-2\nu}}\frac{\lambda^{2\nu}}{(\alpha+\lambda)^{2\nu}}\alpha^{2\nu} \le \alpha^{2\nu} = \varphi(\alpha) \]
for arbitrary $\alpha>0$ and $\lambda>0$.

Thus, \autoref{thCr} yields for every $\nu\in(0,1)$ the equivalence of $\|x_\alpha(y)-x^\dag\|^2=\Ord(\alpha^{2\nu})$ and $e(\lambda)=\Ord(\lambda^{2\nu})$.

\item\label{enTikhonovLog}
Similarly, we also get the equivalence in the case of logarithmic convergence rates. Let $0<\nu<\mu<1$ and define for $\alpha\in(0,\e^{-\frac\nu\mu}]$ the function $\varphi(\alpha) = \left|\log\alpha\right|^{-\nu}$ (for bigger values of $\alpha$, we may simply set $\varphi(\alpha)=\varphi(\e^{-\frac\nu\mu})$). Then, we have
\begin{align*}
(\varphi\tilde r_\alpha^\mu)'(\lambda) &= \frac{\alpha^{2\mu}}{(\alpha+\lambda)^{2\mu+1}\left|\log\lambda\right|^{\nu+1}}\left(\nu\frac{\alpha+\lambda}\lambda-2\mu\left|\log\lambda\right|\right) \\
&\le -\frac{2(\mu\left|\log\lambda\right|-\nu)\alpha^{2\mu}}{(\alpha+\lambda)^{2\mu+1}\left|\log\lambda\right|^{\nu+1}} \le 0
\end{align*}
for all $\lambda\in[\alpha,\e^{-\frac\nu\mu})$. Thus, $\varphi\tilde r_\alpha^\mu$ is decreasing on $[\alpha,\e^{-\frac\nu\mu})$, which implies \eqref{eqQualification} with $A=1$.

So, \autoref{thCr} tells us that $\|x_\alpha(y)-x^\dag\|^2=\Ord(\left|\log\alpha\right|^{-\nu})$ if and only if $e(\lambda)=\Ord(\left|\log\lambda\right|^{-\nu})$.
\end{enumerate}
\end{example}

\section{Convergence Rates for Noisy Data}\label{seConvRateNoisy}
We now want to estimate the distance of the regularised solution $x_\alpha(\tilde y)$ to the minimum-norm solution $x^\dag$ if we do not have the exact data $y$, but only some approximation $\tilde y$ of it.

In this case, we consider the regularisation parameter $\alpha$ as a function of the noisy data~$\tilde y$ such that the distance between $x_\alpha(\tilde y)$ and $x^\dag$ is minimal. Thus, we are interested in the convergence rate of the expression $\inf_{\alpha>0}\|x_\alpha(\tilde y)-x^\dag\|$ to zero as the distance between $\tilde y$ and $y$ tends to zero. We therefore want to find an upper bound for the expression $\sup_{\tilde y\in\bar B_\delta(y)}\inf_{\alpha>0}\|x_\alpha(\tilde y)-x^\dag\|$, where $\bar B_\delta(y)=\{\tilde y\in Y\mid \|\tilde y-y\|\le\delta\}$ denotes the closed ball with radius $\delta>0$ around the data $y$.

Let us first consider the trivial case where $\|x_\alpha(y)-x^\dag\|=0$ for all $\alpha$ in a vicinity of $0$.
\begin{lemma}\label{thConvRateEstimateTrivial}
We use \autoref{noSetting} and assume that there exists an $\varepsilon>0$ such that
\[ \|x_\alpha(y)-x^\dag\|=0 \quad\text{for all}\quad\alpha\in(0,\varepsilon]. \]
Then, we have
\begin{equation}\label{eqConvRateEstimateTrivial}
\sup_{\tilde y\in\bar B_\delta(y)}\inf_{\alpha>0}\|x_\alpha(\tilde y)-x^\dag\|^2 \le \frac{\rho^2}\varepsilon\delta^2,
\end{equation}
where $\rho>0$ is chosen as in \autoref{deGenerator}~\ref{enGeneratorBounded}.
\end{lemma}
\begin{proof}
Let $\tilde y\in\bar B_\delta(y)$ be fixed. Then, using that $Lr_\alpha(L^*L)=r_\alpha(LL^*)L$, it follows from \autoref{deGenerator}~\ref{enGeneratorBounded} that
\begin{equation}\label{eqDistRegSol}
\|x_\alpha(\tilde y)-x_\alpha(y)\|^2 = \left<\tilde y-y,r_\alpha^2(LL^*)LL^*(\tilde y-y)\right> \le \delta^2\max_{\lambda>0}\lambda r_\alpha^2(\lambda) \le \rho^2\frac{\delta^2}\alpha.
\end{equation}
The right hand side is uniform for all $\tilde y\in\bar B_\delta(y)$. Thus, picking $\alpha=\varepsilon$, we get
\[ \sup_{\tilde y\in\bar B_\delta(y)}\inf_{\alpha>0}\|x_\alpha(\tilde y)-x^\dag\|^2 \le \inf_{\alpha>0}\left(\|x_\alpha(y)-x^\dag\|+\rho\frac\delta{\sqrt{\alpha}}\right)^2 \le \frac{\rho^2}\varepsilon\delta^2, \]
which is \eqref{eqConvRateEstimateTrivial}.
\end{proof}

In the general case, we estimate the optimal regularisation parameter $\alpha$ to be in the vicinity of the value~$\alpha_\delta$, which is chosen as the solution of the implicit equation \eqref{eqConvRateOptRegParam} and is therefore only depending on the distance $\delta$ between the correct data $y$ and the noisy data $\tilde y$.

\begin{lemma}\label{thConvRateEstimate}
We use again \autoref{noSetting} and consider the case where $\|x_\alpha(y)-x^\dag\|>0$ for all $\alpha>0$. 

If we choose for every $\delta>0$ the parameter $\alpha_\delta>0$ such that
\begin{equation}\label{eqConvRateOptRegParam}
\alpha_\delta\|x_{\alpha_\delta}(y)-x^\dag\|^2 = \delta^2,
\end{equation}
then there exists a constant $C_1>0$ such that
\begin{equation}\label{eqConvRateUpperBound}
\sup_{\tilde y\in\bar B_\delta(y)}\inf_{\alpha>0}\|x_\alpha(\tilde y)-x^\dag\|^2\le C_1\frac{\delta^2}{\alpha_\delta}\quad\text{for all}\quad\delta>0.
\end{equation}

Moreover, there exists a constant $C_0>0$ such that
\begin{equation}\label{eqConvRateLowerBound}
\sup_{\tilde y\in\bar B_\delta(y)}\inf_{\alpha>0}\|x_\alpha(\tilde y)-x^\dag\|^2 \ge C_0\frac{\delta^2}{\alpha_\delta}
\end{equation}
for all $\delta>0$ which fulfil that $\alpha_\delta\in\sigma(LL^*)$, where $\sigma(LL^*)\subset[0,\infty)$ denotes the spectrum of the operator $LL^*$. 
\end{lemma}
\begin{proof}
First, we remark that the function
\[ A:(0,\infty)\to(0,\infty),\quad A(\alpha)=\alpha\|x_\alpha(y)-x^\dag\|^2=\int_0^{\|L\|^2}\alpha\tilde r_\alpha(\lambda)\d e(\lambda) \]
is, according to \autoref{deGenerator}~\ref{enGeneratorErrorReg} together with the assumption that $\|x_\alpha(y)-x^\dag\|>0$ for all $\alpha>0$, continuous and strictly increasing; and it satisfies $\lim_{\alpha\to0}A(\alpha)=0$ and $\lim_{\alpha\to\infty}A(\alpha)=\infty$.
Therefore, we find for every $\delta>0$ a unique value $\alpha_\delta=A^{-1}(\delta^2)$.

Let $\tilde y\in\bar B_\delta(y)$. Then, as in the proof of \autoref{thConvRateEstimateTrivial}, see~\eqref{eqDistRegSol}, we find that
\[ \|x_\alpha(\tilde y)-x_\alpha(y)\|^2 \le \rho^2\frac{\delta^2}\alpha. \]

From this estimate, we obtain with the triangular inequality and with the definition~\eqref{eqConvRateOptRegParam} of~$\alpha_\delta$ that
\[ \sup_{\tilde y\in\bar B_\delta(y)}\inf_{\alpha>0}\|x_\alpha(\tilde y)-x^\dag\|^2 \le \inf_{\alpha>0}\left(\|x_\alpha(y)-x^\dag\|+\rho\frac\delta{\sqrt{\alpha}}\right)^2 \le (1+\rho)^2\frac{\delta^2}{\alpha_\delta}, \]
which is the upper bound \eqref{eqConvRateUpperBound} with the constant $C_1=(1+\rho)^2$.

For the lower bound \eqref{eqConvRateLowerBound}, we write similarly
\begin{equation}\label{eqConvRateError}
\begin{aligned}
\|x_\alpha(\tilde y)-x^\dag\|^2 &= 
\begin{multlined}[t][0.6\textwidth]
\|x_\alpha(y)-x^\dag\|^2+\|x_\alpha(\tilde y)-x_\alpha(y)\|^2 \\[1ex]
+2\left<x_\alpha(\tilde y)-x_\alpha(y),x_\alpha(y)-x^\dag\right>
\end{multlined} \\
&= 
\begin{multlined}[t][0.6\textwidth]
\|x_\alpha(y)-x^\dag\|^2+\left<\tilde y-y,r_\alpha^2(LL^*)LL^*(\tilde y-y)\right> \\[1ex]
+2\left<r_\alpha(LL^*)(\tilde y-y),r_\alpha(LL^*)LL^*y-y\right>.
\end{multlined}
\end{aligned}
\end{equation}

Now, from the continuity of $\tilde r_{\alpha_\delta}$ and \autoref{deGenerator}~\ref{enGeneratorLimit}, we find that for every $\delta>0$ 
there exists a parameter $a_\delta\in(0,\alpha_\delta)$ such that $\tilde r_{\alpha_\delta}(a_\delta)<\tilde\rho$.

Then, the assumption $\alpha_\delta\in\sigma(LL^*)$ implies that the spectral measure $F$ of the operator~$LL^*$ fulfils 
$F_{[a_\delta,2\alpha_\delta]}\ne0$.

Suppose now that
\begin{equation}\label{eqConvRateProj}
z_\delta = F_{[a_\delta,2\alpha_\delta]}(r_{\alpha_\delta}(LL^*)LL^*y-y) \ne 0.
\end{equation}

Then, choosing $\tilde y=y+\delta\frac{z_\delta}{\|z_\delta\|}$, the equation \eqref{eqConvRateError} becomes
\[ \|x_\alpha(\tilde y)-x^\dag\|^2 = \|x_\alpha(y)-x^\dag\|^2+\frac{\delta^2}{\|z_\delta\|^2}\left<z_\delta,r_\alpha^2(LL^*)LL^*z_\delta\right> \\
+\frac{2\delta}{\|z_\delta\|}\left<r_\alpha(LL^*)z_\delta,z_\delta\right>. \]
Thus, we may drop the last term as it is non-negative, which gives us the lower bound
\[ \sup_{\tilde y\in\bar B_\delta(y)}\inf_{\alpha>0}\|x_\alpha(\tilde y)-x^\dag\|^2 \ge \inf_{\alpha>0}\left(\|x_\alpha(y)-x^\dag\|^2+\delta^2\min_{\lambda\in[a_\delta,2\alpha_\delta]}\lambda r_\alpha^2(\lambda)\right). \]
Since we get from \autoref{deGenerator}~\ref{enGeneratorError} the inequality
\[ \lambda r_\alpha^2(\lambda)=\frac{\big(1-\sqrt{\tilde r_\alpha(\lambda)}\big)^2}\lambda \ge \frac{\big(1-\sqrt{\tilde r_\alpha(a_\delta)}\big)^2}{2\alpha_\delta}\quad\text{for all}\quad \lambda\in[a_\delta,2\alpha_\delta], \]
we can estimate further
\[ \sup_{\tilde y\in\bar B_\delta(y)}\inf_{\alpha>0}\|x_\alpha(\tilde y)-x^\dag\|^2 \ge \inf_{\alpha>0}\left(\|x_\alpha(y)-x^\dag\|^2+\delta^2\frac{\big(1-\sqrt{\tilde r_\alpha(a_\delta)}\big)^2}{2\alpha_\delta}\right). \]
Now, since $\alpha\mapsto\tilde r_\alpha(\lambda)$ is for every $\lambda>0$ increasing, see \autoref{deGenerator}~\ref{enGeneratorErrorReg}, the first term is increasing in $\alpha$, see \eqref{eqCrResidum}, and the second term is decreasing in $\alpha$. Thus, we can estimate the expression for $\alpha<\alpha_\delta$ from below by the second term at $\alpha=\alpha_\delta$, and for $\alpha\ge\alpha_\delta$ by the first term at $\alpha=\alpha_\delta$:
\[ \sup_{\tilde y\in\bar B_\delta(y)}\inf_{\alpha>0}\|x_\alpha(\tilde y)-x^\dag\|^2 \ge 
\min\left\{\|x_{\alpha_\delta}(y)-x^\dag\|^2,\delta^2\frac{\big(1-\sqrt{\tilde r_{\alpha_\delta}(a_\delta)}\big)^2}{2\alpha_\delta}\right\}
\ge \frac{(1-\sqrt{\tilde\rho})^2}2\frac{\delta^2}{\alpha_\delta}, \]
which is \eqref{eqConvRateLowerBound} with $C_0=\frac12(1-\sqrt{\tilde\rho})^2$.

If $z_\delta$, defined by \eqref{eqConvRateProj} happens to vanish, the same argument works with an arbitrary non-zero element $z_\delta\in\mathcal R(F_{[a_\delta,2\alpha_\delta]})$, since then the last term in \eqref{eqConvRateError} is zero for $\tilde y=y+\delta\frac{z_\delta}{\|z_\delta\|}$.
\end{proof}

From \autoref{thConvRateEstimateTrivial} and \autoref{thConvRateEstimate}, we now get an equivalence relation between the noisy and the noise-free convergence rates.

\begin{proposition}\label{thConvRateNoisy}
We use \autoref{noSetting}. Let further $\varphi:[0,\infty)\to[0,\infty)$ be a strictly increasing function satisfying $\varphi(0)=0$ and 
\begin{equation}\label{eqSubhomog}
\varphi(\gamma\alpha)\le g(\gamma)\varphi(\alpha)\quad\text{for all}\quad\alpha>0,\;\gamma>0
\end{equation}
for some increasing function $g:(0,\infty)\to(0,\infty)$.

Moreover, we assume that there exists a constant $C>0$ with
\begin{equation}\label{eqConvRateNoisyCond1}
\frac{\tilde r_\alpha(\lambda)}{\tilde r_\beta(\lambda)} \le C\frac{\varphi(\alpha)}{\varphi(\beta)}\quad\text{for all}\quad0<\alpha\le\beta\le\lambda
\end{equation}
and there is a constant $\tilde C$ such that
\begin{equation}\label{eqConvRateNoisyCond2}
\frac{\tilde r_\alpha(\lambda)}{\tilde r_\beta(\lambda)} \ge \tilde C\frac{\varphi(\alpha)}{\varphi(\beta)}\quad\text{for all}\quad0<\lambda\le\alpha\le\beta.
\end{equation}
We define 
\begin{equation}\label{eqConvRateNoisyFct}
\tilde\varphi(\alpha)=\sqrt{\alpha\varphi(\alpha)}\quad\text{and}\quad\psi(\delta)=\frac{\delta^2}{\tilde\varphi^{-1}(\delta)}.
\end{equation}

Then, the following two statements are equivalent:
\begin{enumerate}
\item
There exists a constant $c>0$ such that
\begin{equation}\label{eqConvRateNoisy}
\sup_{\tilde y\in\bar B_\delta(y)}\inf_{\alpha>0}\|x_\alpha(\tilde y)-x^\dag\|^2\le c\psi(\delta)\quad\text{for all}\quad\delta>0.
\end{equation}
\item
There exists a constant $\tilde c>0$ such that
\begin{equation}\label{eqConvRateNoisyNoisefree}
\|x_\alpha(y)-x^\dag\|^2\le \tilde c\varphi(\alpha)\quad\text{for all}\quad\alpha>0.
\end{equation}
\end{enumerate}
\end{proposition}

\begin{proof}
We first remark that \eqref{eqSubhomog} implies that $\tilde\varphi(\gamma\alpha)\le\sqrt{\gamma g(\gamma)}\tilde\varphi(\alpha)$, and so, by setting $\tilde g(\gamma)=\sqrt{\gamma g(\gamma)}$, $\delta=\tilde\varphi(\alpha)$, and $\tilde\gamma=\tilde g(\gamma)$, we get 
\[ \tilde g^{-1}(\tilde\gamma)\tilde\varphi^{-1}(\delta)\le\tilde\varphi^{-1}(\tilde\gamma\delta).\]
Thus, we have
\begin{equation}\label{eqSubhmogPsi}
\psi(\tilde\gamma\delta) = \frac{\tilde\gamma^2\delta^2}{\tilde\varphi^{-1}(\tilde\gamma\delta)} \le \frac{\tilde\gamma^2\delta^2}{\tilde g^{-1}(\tilde\gamma)\tilde\varphi^{-1}(\delta)} = h(\tilde\gamma)\psi(\delta)
\end{equation}
where $h(\tilde\gamma)=\frac{\tilde\gamma^2}{\tilde g^{-1}(\tilde\gamma)}$.

In the case where $\|x_\alpha(y)-x^\dag\|=0$ for all $\alpha\in(0,\varepsilon]$ for some $\varepsilon>0$, the inequality \eqref{eqConvRateNoisyNoisefree} is trivially fulfilled for some $\tilde c>0$. Moreover, we know from \autoref{thConvRateEstimateTrivial} that then the inequality \eqref{eqConvRateEstimateTrivial} holds, which implies the inequality \eqref{eqConvRateNoisy} for some constant $c>0$, since we have, according to the definition of the function $\psi$, that $\psi(\delta)\ge a\delta^2$ for all $\delta\in(0,\delta_0)$ for some constants $a>0$ and $\delta_0>0$.

Thus, we may assume that $\|x_\alpha(y)-x^\dag\|>0$ for all $\alpha>0$.
\begin{itemize}
\item
Let \eqref{eqConvRateNoisyNoisefree} hold. For arbitrary $\delta>0$ we use the regularisation parameter $\alpha_\delta$ 
defined in \eqref{eqConvRateOptRegParam}. Then, the inequality \eqref{eqConvRateNoisyNoisefree} implies that 
\[ \frac{\delta^2}{\alpha_\delta}\le \tilde c\varphi(\alpha_\delta). \]
Consequently,
\[ \tilde\varphi^{-1}\left(\frac\delta{\sqrt{\tilde c}}\right)\le\alpha_\delta, \]
and therefore, using the inequality \eqref{eqConvRateUpperBound} obtained in \autoref{thConvRateEstimate}, we find with \eqref{eqSubhmogPsi} that
\[ \sup_{\tilde y\in\bar B_\delta(y)}\inf_{\alpha>0}\|x_\alpha(\tilde y)-x^\dag\|^2
\le C_1\frac{\delta^2}{\alpha_\delta} \le C_1\tilde c\psi\left(\frac\delta{\sqrt{\tilde c}}\right) \le C_1\tilde ch(\tfrac1{\sqrt{\tilde c}})\psi(\delta), \]
which is the estimate \eqref{eqConvRateNoisy} with $c=C_1 \tilde ch(\frac1{\sqrt{\tilde c}})$.
\item 
Conversely, if \eqref{eqConvRateNoisy} holds, we choose an arbitrary $\delta>0$ such that $\alpha_\delta$ defined by \eqref{eqConvRateOptRegParam} is in the spectrum $\sigma(LL^*)$. Then, we can use the inequality~\eqref{eqConvRateLowerBound} of \autoref{thConvRateEstimate} to obtain from the condition~\eqref{eqConvRateNoisy} that
\[ C_0\frac{\delta^2}{\alpha_\delta}\le c\psi(\delta). \]
Thus, by the definition of $\psi$, we have
\[ \tilde\varphi^{-1}(\delta)\le \frac c{C_0}\alpha_\delta. \]
So, finally, we get with \eqref{eqSubhomog} that
\[ \|x_{\alpha_\delta}(y)-x^\dag\|^2 = \frac{\delta^2}{\alpha_\delta} \le \frac c{C_0}\varphi\left(\frac c{C_0}\alpha_\delta\right)\le\frac c{C_0}g(\tfrac c{C_0})\varphi(\alpha_\delta), \]
and since this holds for every $\delta$ such that $\alpha_\delta\in\sigma(LL^*)$, we have with $\hat c=\frac c{C_0}g(\tfrac c{C_0})$ that
\begin{equation}\label{eqConvRateNoisySpectrum}
\|x_\alpha(y)-x^\dag\|^2 \le \hat c\varphi(\alpha)\quad\text{for all}\quad\alpha\in\sigma(LL^*).
\end{equation}

Finally, we consider some $\alpha\notin\sigma(LL^*)$, $\alpha<\|L\|^2$, and set 
\begin{align*}
\alpha_-&=\sup\{\tilde\alpha\in\sigma(LL^*)\cup\{0\}\mid\tilde\alpha<\alpha\}\quad\text{and}\\
\alpha_+&=\inf\{\tilde\alpha\in\sigma(LL^*)\mid\tilde\alpha>\alpha\}.
\end{align*}
Then, recalling that $\sigma(L^*L)\setminus\{0\}=\sigma(LL^*)\setminus\{0\}$, see for example~\cite[Problem~61]{Hal74}, we find for $\alpha_->0$ (for $\alpha_-=0$, the first term in the following calculation simply vanishes) that
\begin{align*}
\|x_\alpha(y)-x^\dag\|^2 &= \int_0^{\alpha_-}\tilde r_\alpha(\lambda)\d e(\lambda)+\int_{\alpha_+}^{\|L\|^2}\tilde r_\alpha(\lambda)\d e(\lambda) \\
&\le\|x_{\alpha_-}(y)-x^\dag\|^2\sup_{\lambda\in[0,\alpha_-]}\frac{\tilde r_\alpha(\lambda)}{\tilde r_{\alpha_-}(\lambda)}+\|x_{\alpha_+}(y)-x^\dag\|^2\sup_{\lambda\in[\alpha_+,\|L\|^2]}\frac{\tilde r_\alpha(\lambda)}{\tilde r_{\alpha_+}(\lambda)}.
\end{align*}
Using the conditions \eqref{eqConvRateNoisyCond1} and \eqref{eqConvRateNoisyCond2}, we have with \eqref{eqConvRateNoisySpectrum} that
\[ \|x_\alpha(y)-x^\dag\|^2 \le \frac{\hat c}{\tilde C}\varphi(\alpha_-)\frac{\varphi(\alpha)}{\varphi(\alpha_-)}+C\hat c\varphi(\alpha_+)\frac{\varphi(\alpha)}{\varphi(\alpha_+)} = (C+\tfrac1{\tilde C})\hat c\varphi(\alpha), \]
which is \eqref{eqConvRateNoisyNoisefree} with $\tilde c=(C+\tfrac1{\tilde C})\hat c$.
\end{itemize}
\end{proof}

\begin{remark}
If we consider Tikhonov regularisation, then we can ignore the two conditions \eqref{eqConvRateNoisyCond1} and \eqref{eqConvRateNoisyCond2} in \autoref{thConvRateNoisy} if we have a quadratic upper bound on the function $g$ in \eqref{eqSubhomog}.

Indeed, let $\varphi:[0,\infty)\to[0,\infty)$ be an arbitrary increasing function fulfilling \eqref{eqSubhomog} for some increasing function $g:(0,\infty)\to(0,\infty)$ which is bounded by
\begin{equation}\label{eqConvRateNoisyCondTik}
g(\gamma)\le \frac C4(1+\gamma^2)\quad\text{for all}\quad\gamma>0
\end{equation}
for some constant $C>0$. Then, the conditions \eqref{eqConvRateNoisyCond1} and \eqref{eqConvRateNoisyCond2} are fulfilled for the error function $\tilde r_\alpha$ of Tikhonov regularisation, given by $\tilde r_\alpha(\lambda)=\frac{\alpha^2}{(\alpha+\lambda)^2}$.

To see this, we remark that, for $0<\alpha\le\beta$, the ratio
\[ \frac{\tilde r_\alpha(\lambda)}{\tilde r_\beta(\lambda)} = \left(\frac\alpha\beta\frac{\beta+\lambda}{\alpha+\lambda}\right)^2 \]
is decreasing in $\lambda$. Therefore, for $\lambda\ge\beta$, we get that
\[ \frac{\tilde r_\alpha(\lambda)}{\tilde r_\beta(\lambda)} \le \frac4{(1+\frac\beta\alpha)^2} \le \frac4{1+(\frac\beta\alpha)^2} \le \frac C{g(\frac\beta\alpha)} \le C\frac{\varphi(\alpha)}{\varphi(\beta)}, \]
which is \eqref{eqConvRateNoisyCond1}.

We similarly find for $\lambda\le\alpha$ that
\[ \frac{\tilde r_\alpha(\lambda)}{\tilde r_\beta(\lambda)} \ge \frac14\left(1+\frac\alpha\beta\right)^2  \ge \frac14 \ge \frac14\frac{\varphi(\alpha)}{\varphi(\beta)}, \]
which is \eqref{eqConvRateNoisyCond2} with $\tilde C=\frac14$.
\end{remark}

We want to apply this theorem now to the two special cases discussed previously in \autoref{exTikhonov}.
\begin{example}
\begin{enumerate}
\item
In the case of \autoref{exTikhonov}~\ref{enTikhonovClassical}, where we considered Tikhonov regularisation with a convergence rate given by $\varphi(\alpha)=\alpha^{2\nu}$ for some $\nu\in(0,1)$, the condition \eqref{eqSubhomog} in \autoref{thConvRateNoisy} is clearly fulfilled with $g(\gamma)=\gamma^{2\nu}$. In particular, $g$ satisfies that $g(\gamma)\le 1+\gamma^2$, which is \eqref{eqConvRateNoisyCondTik} with $C=4$, and thus the conditions \eqref{eqConvRateNoisyCond1} and \eqref{eqConvRateNoisyCond2} in \autoref{thConvRateNoisy} follow as in the remark above.

So, we can apply \autoref{thConvRateNoisy} and it only remains to calculate
\[ \tilde\varphi^{-1}(\delta) = \delta^{\frac2{2\nu+1}}\quad\text{and}\quad\psi(\delta) = \delta^{2-\frac2{2\nu+1}} = \delta^{\frac{4\nu}{2\nu+1}}. \]
Thus, we recover the classical result, see \cite[Theorem 2.6]{Neu97}, that the convergence rate $\|x_\alpha(y)-x^\dag\|^2=\mathcal O(\alpha^{2\nu})$ for the correct data $y$ is equivalent to the convergence rate $\sup_{\tilde y\in\bar B_\delta(y)}\inf_{\alpha>0}\|x_\alpha(\tilde y)-x^\dag\|^2=\mathcal O(\delta^{\frac{4\nu}{2\nu+1}})$ for noisy data.
\item
Next, we look at Tikhonov regularisation with the logarithmic convergence rate 
\begin{equation*}
\varphi(\alpha)= \begin{cases}
                          \left|\log\alpha\right|^{-\nu} & \text{if}\;0 < \alpha < \e^{-(1+\nu)},\\
                          (1+\nu)^{-\nu} &\text{if}\;\alpha\ge\e^{-(1+\nu)},
                         \end{cases}
\end{equation*}
see \autoref{exTikhonov}~\ref{enTikhonovLog}.
First, we remark that $\varphi$ is concave. This is because $\varphi$ is increasing, constant for $\alpha> \e^{-(1+\nu)}$, and for $0 < \alpha<\e^{-(1+\nu)}$ we have
\[ \varphi''(\alpha) = \frac\nu{\alpha^2}\left|\log\alpha\right|^{-(\nu+2)}\left(1+\nu-\left|\log\alpha\right|\right) < 0. \]
Therefore, and because $\varphi(0)=0$, we have
\[ \varphi(\gamma\alpha) \le \gamma\varphi(\alpha)\quad\text{for all}\quad\gamma\ge1,\;\alpha>0. \]
Thus, using that $\varphi$ is increasing, the requirement \eqref{eqSubhomog} in \autoref{thConvRateNoisy} 
is fulfilled with
\[ g(\gamma) = \begin{cases}1&\text{if}\;\gamma<1, \\\gamma&\text{if}\;\gamma\ge1.\end{cases} \]
In particular, this function $g$ satisfies the inequality \eqref{eqConvRateNoisyCondTik} with $C=4$ and therefore, also the conditions \eqref{eqConvRateNoisyCond1} and \eqref{eqConvRateNoisyCond2} in \autoref{thConvRateNoisy} are fulfilled according to the previous remark.

To get the corresponding function $\psi$, as defined in \eqref{eqConvRateNoisyFct}, we have to solve the implicit equation $\delta = \tilde\varphi(\frac{\delta^2}{\psi(\delta)})$, where $\tilde\varphi$ is defined in \eqref{eqConvRateNoisyFct} and with the specific choice of $\varphi(\alpha) =\left|\log\alpha\right|^{-\nu}$ for $\alpha<\e^{-(1+\nu)}$ satisfies $\tilde\varphi^2(\alpha)=\alpha\left|\log\alpha\right|^{-\nu}$.
This equation then reads as follows:
\begin{equation}\label{eqConvRateNoisyTikLog}
\psi(\delta) = \left|\log\frac{\delta^2}{\psi(\delta)}\right|^{-\nu}.
\end{equation}
By solving this equation for $\delta$, we get
\[ \delta=\sqrt{\psi(\delta)}\exp\left(-\frac1{2\psi^{\frac1\nu}(\delta)}\right), \]
which, in particular, shows that the function $\psi$ is increasing and furthermore, because of $\lim_{\delta\downarrow0}\psi(\delta)=0$, $\psi(\delta)<1$ for sufficiently small $\delta>0$.
Therefore, we find for small $\delta>0$ that
\begin{equation}\label{eqConvRateNoisyFctTikLog}
\delta\le \exp\left(-\frac1{2\psi^{\frac1\nu}(\delta)}\right),\quad\text{that is}\quad \psi(\delta)\ge|2\log\delta|^{-\nu}.
\end{equation}

Moreover, if we write $\psi$ as 
\[ \psi(\delta) = \left|\log\delta\right|^{-\nu}f(\delta) \]
for some function $f$, the implicit equation \eqref{eqConvRateNoisyTikLog} becomes
\[ f(\delta)=\left|\frac{\log\delta}{\log(f(\delta))-2\log\delta-\log(|\log\delta|^\nu)}\right|^\nu. \]
Since $\lim_{\delta\downarrow0}\frac{\log(\left|\log\delta\right|^\nu)}{\log\delta}=0$, we find parameters $\varepsilon\in(0,1)$ and $\delta_0\in(0,1)$ such that we have for all $\delta<\delta_0$ the inequality $0\le\log(\left|\log\delta\right|^\nu)\le\varepsilon\left|\log\delta\right|$. Assuming that $f(\delta)\ge1$ gives 
\[ f(\delta)\le\left(\frac{\left|\log\delta\right|}{\log(f(\delta))+(2-\varepsilon)\left|\log\delta\right|}\right)^\nu \le\frac1{(2-\varepsilon)^\nu} < 1, \]
which is a contradiction to the assumption. Thus, $f(\delta) <1$. 

Since we know already from \eqref{eqConvRateNoisyFctTikLog} that $f(\delta) \ge 2^{-\nu}$, it therefore follows from \autoref{thConvRateNoisy} that the convergence rate $\|x_\alpha(y)-x^\dag\|^2=\mathcal O(\left|\log\alpha\right|^{-\nu})$ is equivalent to $\sup_{\tilde y\in\bar B_\delta(y)}\inf_{\alpha>0}\|x_\alpha(\tilde y)-x^\dag\|^2=\mathcal O(\left|\log\delta\right|^{-\nu})$.
\end{enumerate}
\end{example}

\section{Relation to Variational Inequalities}\label{seVarIneq}
Instead of characterising the convergence rate of the regularised solution via the behaviour of the spectral decomposition of the minimum-norm solution $x^\dag$, we may also check variational inequalities for the element $x^\dag$, see \cite{HofKalPoeSch07,SchGraGroHalLen09,HeiHof09,SchuKalHofKaz12}. In \cite{AndElbHooQiuSch15}, it was shown that for Tikhonov regularisation and convergence rates of the order $\mathcal O(\alpha^{2\nu})$, $\nu\in(0,1)$, such variational inequalities are equivalent to specific convergence rates.

In this section, we generalise this result to cover general regularisation methods and convergence rates.

\begin{proposition}\label{thVarIneq}
We consider again the setting of \autoref{noSetting}.
Moreover, let $\varphi:[0,\infty)\to[0,\infty)$ be an increasing, continuous function and $\nu\in(0,1)$.

Then, the following two statements are equivalent:
\begin{enumerate}
\item
There exists a constant $C>0$ with
\begin{equation}\label{eqVarIneqSpectralTail}
e(\lambda) \le C\varphi^{2\nu}(\lambda)\quad\text{for all}\quad\lambda>0.
\end{equation}
\item
There exists a constant $\tilde C>0$ such that
\begin{equation}\label{eqVarIneq}
\left<x^\dag,x\right> \le \tilde C\|\varphi(L^*L)x\|^\nu\|x\|^{1-\nu}\quad\text{for all}\quad x\in X.
\end{equation}
\end{enumerate}
\end{proposition}
\begin{proof}
\begin{itemize}
\item
Assume first that \eqref{eqVarIneq} holds. Then, we have for all $\lambda > 0$
\begin{align*}
\|E_{[0,\lambda]}x^\dag\|^2 &= \left<x^\dag,E_{[0,\lambda]}x^\dag\right> \\
&\le \tilde C\|\varphi(L^*L)E_{[0,\lambda]}x^\dag\|^\nu\|E_{[0,\lambda]}x^\dag\|^{1-\nu} \\
&\le \tilde C\varphi^\nu(\lambda)\|E_{[0,\lambda]}x^\dag\|,
\end{align*}
which implies \eqref{eqVarIneqSpectralTail} with $C = \tilde C^2$.

\item
On the other hand, if \eqref{eqVarIneqSpectralTail} is fulfilled, then we can estimate for arbitrary $\Lambda>0$ and every $x\in X$
\begin{equation}\label{eqVarIneqLow}
\left|\left<E_{[0,\Lambda]}x^\dag,x\right>\right| \le \|E_{[0,\Lambda]}x^\dag\|\|x\|\le \sqrt C\,\varphi^\nu(\Lambda)\|x\|.
\end{equation}

Furthermore, we get with the bounded, invertible operator $T=\varphi(L^*L)|_{\mathcal R(E_{[\Lambda,\infty)})}$ that
\begin{equation}\label{eqVarIneqHighSpectrum}
\begin{aligned}
\left|\left<E_{[\Lambda,\infty)}x^\dag,x\right>\right| &= \left|\left<T^{-1}E_{[\Lambda,\infty)}x^\dag,TE_{[\Lambda,\infty)}x\right>\right| \\
&\le \|TE_{[\Lambda,\infty)}x\|\sqrt{\lim_{\varepsilon\downarrow0}\int_{\Lambda-\varepsilon}^{\|L\|^2}\frac1{\varphi^2(\lambda)}\d e(\lambda)}.
\end{aligned}
\end{equation}
Integrating by parts, we can rewrite the integral in the form
\[ \int_{\Lambda-\varepsilon}^{\|L\|^2}\frac1{\varphi^2(\lambda)}\d e(\lambda) = \frac{e(\|L\|^2)}{\varphi^2(\|L\|^2)}-\frac{e(\Lambda-\varepsilon)}{\varphi^2(\Lambda-\varepsilon)}+2\int_{\Lambda-\varepsilon}^{\|L\|^2}\frac{e(\lambda)}{\varphi^3(\lambda)}\d\varphi(\lambda). \]
Using now \eqref{eqVarIneqSpectralTail} and dropping all negative terms, we arrive at
\[ \lim_{\varepsilon\downarrow0}\int_{\Lambda-\varepsilon}^{\|L\|^2}\frac1{\varphi^2(\lambda)}\d e(\lambda)\le \frac C{\varphi^{2-2\nu}(\|L\|^2)}+\frac C{1-\nu}\frac1{\varphi^{2-2\nu}(\Lambda)} \le \frac{c^2}{\varphi^{2-2\nu}(\Lambda)} \]
with the constant $c>0$ given by $c^2=C(1+\frac1{1-\nu})$. Plugging this into \eqref{eqVarIneqHighSpectrum}, we find that
\begin{equation}\label{eqVarIneqHigh}
\left|\left<E_{[\Lambda,\infty)}x^\dag,x\right>\right| \le \frac c{\varphi^{1-\nu}(\Lambda)}\|\varphi(L^*L)x\|.
\end{equation}

We now pick
\[ \Lambda = \inf\{\lambda>0\mid \left|\left<E_{[0,\lambda]}x^\dag,x\right>\right|\ge\tfrac12\left|\left<x^\dag,x\right>\right|\} \]
and assume that $\Lambda>0$; otherwise $\left<x^\dag,x\right>=0$ and \eqref{eqVarIneq} is trivially fulfilled. Then, the right continuity of $\lambda\mapsto\left<E_{[0,\lambda]}x^\dag,x\right>$ implies that
\[ \left|\left<E_{[0,\Lambda]}x^\dag,x\right>\right|\ge\frac12\left|\left<x^\dag,x\right>\right|. \]
Moreover, we have that
\[ \left|\left<E_{[\lambda,\infty)}x^\dag,x\right>\right|\ge\left|\left<x^\dag,x\right>\right|-\left|\left<E_{[0,\lambda]}x^\dag,x\right>\right|>\frac12\left|\left<x^\dag,x\right>\right| \]
for every $\lambda\in(0,\Lambda)$. Therefore, the left continuity of $\lambda\mapsto\left<E_{[\lambda,\infty)}x^\dag,x\right>$ implies that
\[ \left|\left<E_{[\Lambda,\infty)}x^\dag,x\right>\right|\ge\frac12\left|\left<x^\dag,x\right>\right|. \]

Thus, we get with the estimates \eqref{eqVarIneqLow} and \eqref{eqVarIneqHigh} that
\begin{align*}
\left<x^\dag,x\right> &\le 2\left|\left<E_{[0,\Lambda]}x^\dag,x\right>\right|^{1-\nu}\left|\left<E_{[\Lambda,\infty)}x^\dag,x\right>\right|^\nu \\
&\le 2C^{\frac{1-\nu}2}c^\nu\|\varphi(L^*L)x\|^\nu\|x\|^{1-\nu}.
\end{align*}
\end{itemize}
\end{proof}

We remark that the first part of this proof also works in the limit case $\nu=1$, which shows that \eqref{eqVarIneq} implies \eqref{eqVarIneqSpectralTail} also for $\nu=1$.
\begin{corollary}\label{thSsc}
We use again \autoref{noSetting}. Let further $\varphi:[0,\infty)\to[0,\infty)$ be an increasing, continuous function and $\nu\in(0,1]$.

Then, the standard source condition
\begin{equation}\label{eqSsc}
x^\dag \in \mathcal R(\varphi^\nu(L^*L))
\end{equation}
implies the variational inequality
\begin{equation}\label{eqSscVarIneq}
\left<x^\dag,x\right> \le C\|\varphi(L^*L)x\|^\nu\|x\|^{1-\nu}\quad\text{for all}\quad x\in X.
\end{equation}
for some constant $C>0$.

Conversely, the variational inequality \eqref{eqSscVarIneq} implies that
\[ x^\dag\in\mathcal R(\psi(L^*L)) \]
for every continuous function $\psi:[0,\infty)\to[0,\infty)$ with $\psi\ge c\varphi^\mu$ for some constants $c>0$ and some $\mu\in(0,\nu)$.
\end{corollary}

\begin{proof}
\begin{itemize}
\item
If $x^\dag$ fulfils \eqref{eqSsc}, then there exists an element $\omega\in X$ with
\begin{equation}\label{eqSscToVarIneq}
\left<x^\dag,x\right> = \left<\omega,\varphi^\nu(L^*L)x\right> \le \|\omega\|\|\varphi^\nu(L^*L)x\|.
\end{equation}
Using the interpolation inequality, see for example \cite[Chapter 2.3]{EngHanNeu96}, we find
\[ \left<x^\dag,x\right> \le \|\omega\|\|\varphi(L^*L)x\|^\nu\|x\|^{1-\nu}, \]
which is \eqref{eqSscVarIneq} with $C=\|\omega\|$.
\item
If, on the other hand, \eqref{eqSscVarIneq} holds, then, according to \autoref{thVarIneq}, there exists a constant $\tilde C>0$ such that $e(\lambda)\le\tilde C\varphi^{2\nu}(\lambda)$. Now, similarly to the proof of \autoref{thVarIneq} we get with $T=\psi(L^*L)|_{\mathcal R(E_{(\Lambda,\infty)})}$ that
\[ \left<E_{(\Lambda,\infty)}x^\dag,x\right> \le \left<T^{-1}E_{(\Lambda,\infty)}x^\dag,TE_{(\Lambda,\infty)}x\right>\le \|TE_{(\Lambda,\infty)}x\|\sqrt{\int_\Lambda^{\|L\|^2}\frac1{\psi^2(\lambda)}\d e(\lambda)}, \]
and, using the lower bound on $\psi$, that
\[ \int_\Lambda^{\|L\|^2}\frac1{\psi^2(\lambda)}\d e(\lambda) \le \frac1{c^2}\int_\Lambda^{\|L\|^2}\frac1{\varphi^{2\mu}(\lambda)}\d e(\lambda)
\le \tilde c^2\varphi^{2(\nu-\mu)}(\|L\|^2), \]
for some constant $\tilde c>0$.
So,
\[ \left<x^\dag,x\right> = \lim_{\Lambda\to0}\left<E_{(\Lambda,\infty)}x^\dag,x\right> \le \tilde c\varphi^{\nu-\mu}(\|L\|^2)\|\psi(L^*L)x\| \]
which implies that $x^\dag\in\mathcal R(\psi(L^*L))$, see for example \cite[Lemma 8.21]{SchGraGroHalLen09}.
\end{itemize}
\end{proof}

\begin{remark}
In general, the inequality \eqref{eqSscVarIneq} does not imply the standard source condition \eqref{eqSsc}. Let us for example consider the case where we have an increasing, continuous function $\varphi:[0,\infty)\to[0,\infty)$ with $\varphi(0)=0$, $\varphi(\lambda)>0$ for all $\lambda>0$, and
\[ c\varphi^{2\nu}(\lambda)\le e(\lambda)\le C\varphi^{2\nu}(\lambda)\quad\text{for all}\quad \lambda>0 \]
for some constants $0<c\le C$. 

Now, the standard source condition \eqref{eqSsc} would imply that we can find a $\xi\in\mathcal N(L)^\perp$ with $x^\dag=\varphi^\nu(L^*L)\xi$. Thus, we would get with $T=\varphi^\nu(L^*L)|_{\mathcal R(E_{(\Lambda,\infty)})}$ that
\[ \|\xi\|^2 = \lim_{\Lambda\to0}\|E_{(\Lambda,\infty)}\xi\|^2 = \lim_{\Lambda\to0}\|T^{-1}E_{(\Lambda,\infty)}x^\dag\|^2 = \lim_{\Lambda\to0}\int_\Lambda^{\|L\|^2}\frac1{\varphi^{2\nu}(\lambda)}\d e(\lambda). \]
However, in the limit $\Lambda\to0$, we have that
\begin{align*}
\int_\Lambda^{\|L\|^2}\frac1{\varphi^{2\nu}(\lambda)}\d e(\lambda) &= \frac{e(\|L\|^2)}{\varphi^{2\nu}(\|L\|^2)}-\frac{e(\Lambda)}{\varphi^{2\nu}(\Lambda)}+2\nu\int_\Lambda^{\|L\|^2}\frac{e(\lambda)}{\varphi^{2\nu+1}(\lambda)}\d\varphi(\lambda) \\
& \ge c-C+2\nu c\log\left(\frac{\varphi(\|L\|^2)}{\varphi(\Lambda)}\right)\to\infty,
\end{align*}
which is a contradiction to the existence of such a point $\xi$.
\end{remark}

\section{Connection to Approximate Source Conditions}\label{seApproxSourcCond}
Another approach to weaken the standard source condition~\eqref{eqSsc} to obtain a condition which is equivalent to the convergence rate was introduced in~\cite{HofMat07}, see also \cite{FleHofMat11}. The idea is that for the argument~\eqref{eqSscToVarIneq}, which shows that the standard source condition~\eqref{eqSsc} implies the variational inequality \eqref{eqSscVarIneq}, it would have been enough to be able to approximate the minimum-norm solution $x^\dag$ by a bounded sequence in $\mathcal R(\varphi^\nu(L^*L))$. And the smaller the bound on the sequence, the smaller the constant $C$ in the variational inequality \eqref{eqSscVarIneq} will be. Therefore, the distance between $x^\dag$ and $\mathcal R(\varphi^\nu(L^*L))\cap\bar B_R(0)$ as a function of the radius $R$ of the closed ball $\bar B_R(0)=\{x\in X\mid\|x\|\le R\}$ should be directly related to the convergence rate.
\begin{definition}
In the setting of \autoref{noSetting}, we define the distance function~$d_\varphi$ of a continuous function~$\varphi:[0,\infty)\to[0,\infty)$ by
\begin{equation}\label{eqDistanceFct}
d_\varphi(R) = \inf_{\xi\in\bar B_R(0)}\|x^\dag-\varphi(L^*L)\xi\|.
\end{equation}
\end{definition}

Indeed, this distance function gives us directly an upper bound on the error between the regularised solution $x_\alpha(y)$ and the minimum-norm solution $x^\dag$, see \cite[Theorem~5.5]{HofMat07} or \cite[Proposition~2]{FleHofMat11}. For convenience, we repeat the argument here.
\begin{lemma}\label{thDistanceFctError}
We use \autoref{noSetting} and assume that $\varphi:[0,\infty)\to[0,\infty)$ is an increasing, continuous function with $\varphi(0)=0$ so that there exists a constant $A>0$ such that the inequality
\begin{equation}\label{eqDistanceFctQualification}
\sqrt{\tilde r_\alpha(\lambda)}\varphi(\lambda) \le A\varphi(\alpha)\quad\text{for all}\quad \lambda>0
\end{equation}
holds for every $\alpha>0$. 

Then, we have for every $\xi\in X$ that
\begin{equation}\label{eqDistanceFctError}
\|x_\alpha(y)-x^\dag\| \le \|x^\dag-\varphi(L^*L)\xi\|+A\varphi(\alpha)\|\xi\|\quad\text{for all}\quad\alpha>0.
\end{equation}
\end{lemma}

\begin{proof}
For every vector $\xi\in X$, we find from \eqref{eqErrorExact} with the definition \eqref{eqGeneratorError} of the error function $\tilde r_\alpha$ that
\[ \|x_\alpha(y)-x^\dag\| = \|\tilde r_\alpha^{\frac12}(L^*L)x^\dag\| \le \|\tilde r_\alpha^{\frac12}(L^*L)(x^\dag-\varphi(L^*L)\xi)\|+\|\tilde r_\alpha^{\frac12}(L^*L)\varphi(L^*L)\xi\|. \]
Now, since $\tilde r_\alpha(\lambda)\le1$, we have that $\|\tilde r_\alpha(L^*L)\|\le 1$. Moreover, with $e_\xi(\lambda)=\|E_{(0,\lambda]}\xi\|^2$, we get from the inequality \eqref{eqDistanceFctQualification} that
\[ \|\tilde r_\alpha^{\frac12}(L^*L)\varphi(L^*L)\xi\|^2 = \int_0^{\|L\|^2}\tilde r_\alpha(\lambda)\varphi^2(\lambda)\d e_\xi(\lambda) \le A^2\varphi^2(\alpha)\|\xi\|^2. \]
So, putting the two inequalities together, we obtain \eqref{eqDistanceFctError}.
\end{proof}

Thus, taking the infimum over all $\xi\in\bar B_R(0)$ in \eqref{eqDistanceFctError}, the error $\|x_\alpha(y)-x^\dag\|$ can be bound by a combination of $d_\varphi(R)$ and $\varphi(\alpha)R$. By balancing these terms, we obtain from a given distance function $d_\varphi$ the corresponding convergence rate.

Conversely, we can also show that an upper bound on the spectral projections of the minimum-norm solution gives us an upper bound on the distance function, which then yields another equivalent characterisation for the convergence rate of the regularisation method.

\begin{proposition}\label{thDistanceFctCr}
We use \autoref{noSetting} and assume that $\varphi:[0,\infty)\to[0,\infty)$ is an increasing, continuous function with $\varphi(0)=0$ so that there exists a constant $A>0$ with
\begin{equation}\label{eqDistanceFctCrQualification}
\sqrt{\tilde r_\alpha(\lambda)}\varphi(\lambda) \le A\varphi(\alpha)\quad\text{for all}\quad \lambda>0,\;\alpha>0.
\end{equation}
Moreover, let $d_\varphi$ be the distance function of $\varphi$, and let $\nu\in(0,1)$ be arbitrary.

Then, the following statements are equivalent:
\begin{enumerate}
\item
There exists a constant $C>0$ so that
\begin{equation}\label{eqDistanceFctCrSpectralDecay}
e(\lambda) \le C\varphi^{2\nu}(\lambda)\quad\text{for all}\quad \lambda>0.
\end{equation}
\item
There exists a constant $\tilde C>0$ so that
\begin{equation}\label{eqDistanceFctCr}
d_\varphi(R) \le \tilde CR^{-\frac\nu{1-\nu}}\quad\text{for all}\quad R>0.
\end{equation}
\end{enumerate}
\end{proposition}
\begin{proof}
\begin{itemize}
\item
Assume first that \eqref{eqDistanceFctCr} holds. Then, from \autoref{thDistanceFctError}, we get by taking the infimum of~\eqref{eqDistanceFctError} over all $\xi\in\bar B_R(0)$ for an arbitrary $R>0$ that
\[ \|x_\alpha(y)-x^\dag\| \le d_\varphi(R)+A\varphi(\alpha)R \le \tilde CR^{-\frac\nu{1-\nu}}+A\varphi(\alpha)R. \]
Since the first term is decreasing and the second term is increasing in $R$, we pick for~$R$ the value $R(\alpha)$ given by
\[ R^{-\frac\nu{1-\nu}}(\alpha)=\varphi(\alpha)R(\alpha),\quad\text{that is}\quad R(\alpha) = \varphi^{-(1-\nu)}(\alpha). \]
Thus, we end up with
\[ \|x_\alpha(y)-x^\dag\|\le(\tilde C+A)\varphi^\nu(\alpha). \]
Applying now \autoref{thCr} with the function $\varphi$ therein replaced by $\varphi^{2\nu}$ (we remark that the condition \eqref{eqQualification} is then fulfilled with $\mu=\nu$, since \eqref{eqDistanceFctCrQualification} implies $\varphi^{2\nu}(\lambda)\tilde r_\alpha^\nu(\lambda)\le A^{2\nu}\varphi^{2\nu}(\alpha)$), we find that there exists a constant $C>0$ so that~\eqref{eqDistanceFctCrSpectralDecay} holds.
\item
Conversely, if we have the relation~\eqref{eqDistanceFctCrSpectralDecay}, then we define for arbitrary $\alpha>0$ with the operator $T=\varphi(L^*L)|_{\mathcal R(E_{(\alpha,\infty)})}$ the element
\[ \xi_\alpha = T^{-1}E_{(\alpha,\infty)}x^\dag. \]
Now, the distance of $\varphi(L^*L)\xi_\alpha$ to the minimum-norm solution $x^\dag$ can be estimated according to \eqref{eqDistanceFctCrSpectralDecay} by
\begin{equation}\label{eqDistanceFctValue}
\|x^\dag-\varphi(L^*L)\xi_\alpha\|^2 = \|E_{[0,\alpha]}x^\dag\|^2 \le C\varphi^{2\nu}(\alpha).
\end{equation}

Moreover, we can get an upper bound on the norm of $\xi_\alpha$ by
\[ \|\xi_\alpha\|^2 = \int_\alpha^{\|L\|^2}\frac1{\varphi^2(\lambda)}\d e(\lambda) = \frac{e(\|L\|^2)}{\varphi^2(\|L\|^2)}-\frac{e(\alpha)}{\varphi^2(\alpha)}+2\int_\alpha^{\|L\|^2}\frac{e(\lambda)}{\varphi^3(\lambda)}\d\varphi(\lambda). \]
Using assumption \eqref{eqDistanceFctCrSpectralDecay}, evaluating the integral, and dropping the resulting 
two negative terms, we find that
\begin{equation}\label{eqDistanceFctNorm}
\|\xi_\alpha\|^2\le \frac C{\varphi^{2-2\nu}(\|L\|^2)}+\frac C{1-\nu}\frac1{\varphi^{2-2\nu}(\alpha)} \le \frac{c^2}{\varphi^{2-2\nu}(\alpha)}
\end{equation}
with $c^2=C(1+\frac1{1-\nu})$.

So, combining \eqref{eqDistanceFctValue} and~\eqref{eqDistanceFctNorm}, we have by definition \eqref{eqDistanceFct} 
of the distance function $d_\varphi$ with $R=c\varphi^{-(1-\nu)}(\alpha)$ that  
\[ d_\varphi(c\varphi^{-(1-\nu)}(\alpha)) \le \sqrt C\varphi^\nu(\alpha), \]
and thus it follows by switching to the variable $R$ that
\[ d_\varphi(R) \le \tilde C R^{-\frac\nu{1-\nu}}, \]
where $\tilde C=\sqrt C\,c^{\frac\nu{1-\nu}}$.
\end{itemize}

\end{proof}

\section*{Conclusion}

In this paper, we have proven optimal convergence rates results for regularisation methods for solving linear ill-posed 
operator equations in Hilbert spaces. The result generalises existing convergence rates results on optimality 
of \cite{Neu97} to general source conditions, such as logarithmic source conditions. The results state that 
convergence rates results of regularised solution require a certain decay of the solution in terms of the spectral 
decomposition. Moreover, we also provide optimality results under variational source conditions, extending the results of \cite{AndElbHooQiuSch15}. It is interesting to note that variational source conditions are equivalent to convergence rates 
of the regularised solutions, while the classical results are not. Moreover, we also show that rates of the distance function 
developed in \cite{HofMat07,FleHofMat11} are equivalent to convergence rates of the regularised solutions.

\section*{Acknowledgements}
VA is supported by CNPq grant 201644/2014-2. The work of OS has been supported by the Austrian Science Fund (FWF)
within the national research network Geometry and Simulation, project S11704 (Variational Methods for Imaging on Manifolds) and Project P26687-N25 
(Interdisciplinary Coupled Physics Imaging).

\section*{References}
\printbibliography[heading=none]

\end{document}